\documentclass[leqno]{amsart}
\usepackage{amsmath}
\usepackage{amssymb}
\usepackage{amsthm}
\usepackage{enumerate}
\usepackage[mathscr]{eucal}
\theoremstyle{plain}
\newtheorem{theorem}{Theorem}[section]

\newtheorem{lemma}{Lemma}[section]

\theoremstyle{definition}

\begin{document}
\title[On symmetry of Birkhoff-James orthogonality of linear operators]{On symmetry of Birkhoff-James orthogonality of linear operators on finite-dimensional real Banach spaces} 
\author[ Debmalya Sain, Puja Ghosh \and Kallol Paul ]{Debmalya Sain, Puja Ghosh \and Kallol Paul }

\newcommand{\acr}{\newline\indent}

\address{\llap{\,}Department of Mathematics\acr
                             Jadavpur University\acr
                              Kolkata 700032\acr
                              West Bengal\acr
                              INDIA}
\email{saindebmalya@gmail.com; ghosh.puja1988@gmail.com; kalloldada@gmail.com }

\thanks{The first author feels Honoured to acknowledge the tremendous positive contribution of Prof. Kallol Paul, who happens to be a co-author of the present paper, in every sphere of his life, mathematical or otherwise! He also feels elated to acknowledge the ever helpful nature of Prof. Vladimir Kadets, despite his extremely busy schedule. } 



\begin{abstract}
We  characterize left symmetric linear operators  on a finite dimensional strictly convex and smooth real normed linear  space $ \mathbb{X},$ which answers a question raised recently by one of the authors in \cite{S} [D. Sain, \textit{Birkhoff-James orthogonality of linear operators on finite dimensional Banach spaces, Journal of Mathematical Analysis and Applications, accepted, $ 2016 $}]. We prove that $  T\in B(\mathbb{X}) $ is left symmetric if and only if $ T $ is the zero operator. If $ \mathbb{X} $ is two-dimensional then the same characterization can be obtained without the smoothness assumption. We also explore the properties of right symmetric linear operators defined on a finite dimensional real Banach space. In particular, we prove that smooth linear operators on a finite-dimensional strictly convex and smooth real Banach space can not be right symmetric. 

\end{abstract}

\maketitle
\section{Introduction}
The principal purpose of the present paper is to answer a question raised very recently in \cite{S},  regarding Birkhoff-James orthogonality of linear operators. We also explore other related questions in order to obtain a better description of the symmetry of Birkhoff-James orthogonality of linear operators on finite-dimensional Banach spaces. Let us now briefly establish the relevant notations and terminologies. For a more detailed treatment of Birkhoff-James orthogonality, we refer the readers to the classic works \cite{B, Ja, J} and to some of the more recent works \cite{GSP, PSJ, SP, SPH}. \\

Let $(\mathbb{X}, \|.\|)$ be a normed linear space. In this paper, we would always consider $ \mathbb{X} $ to be over $ \mathbb{R}. $  For any two elements $ x , y$ in $\mathbb{X}$, $x$ is said to be orthogonal to $y$ in the sense of Birkhoff-James\cite{B, J}, written as $x\bot_B y,$ if and only if $\|x\| \leq \|x + \lambda y\|$ for all $\lambda \in \mathbb{R} $.  Birkhoff-James orthogonality is related to many important geometric properties of normed linear spaces, including strict convexity, uniform convexity and smoothness. Let $B(\mathbb{X})$ denote the Banach algebra of all bounded linear operators from $ \mathbb{X} $ to $ \mathbb{X} $. $ T \in B(\mathbb{X}) $ is said to attain norm at $ x \in S_\mathbb{X} $ if $ \| Tx\| = \| T \|. $ Let $M_T$ denote the set of all unit vectors in $S_\mathbb{X}$ at which $T$ attains norm, i.e., $M_T = \{ x \in S_\mathbb{X} \colon \|Tx\| = \|T\| \}.$ For any two elements $ T,A \in  B(\mathbb{X}), $ $ T $ is said to be orthogonal to $ A, $ in the sense of Birkhoff-James, written as $ T \perp_B A, $ if and only if 
  \[ \|T\| \leq \|T+ \lambda A\| ~\forall \lambda \in \mathbb R. \] \\
\noindent James \cite{Ja} proved that Birkhoff-James orthogonality is symmetric in a normed linear space $\mathbb{X}$ of three or more dimensions if and only if a compatible inner product can be defined on $\mathbb{X}$. Since $B(\mathbb{X})$ is not an inner product space, it is interesting to  study the symmetry of Birkhoff-James orthogonality of operators in $B(\mathbb{X})$. It is very easy to observe that in $B(\mathbb{X})$, $T\bot_B A $ may not imply $A \not\perp_B T$ or conversely. Consider T = \( \left( 
 \begin{array}{ccc}
  1 & 0 & 0 \\ 
  0 &  1/2 &  0 \\
  0 &  0 &  1/2 
  \end{array} 
  \right) \) and A = \( \left( 
 \begin{array}{ccc}
  0 & 0 & 0 \\ 
  0 &  1 &  0 \\
  0 &  0 &  0 
  \end{array} 
  \right) \) on $(\mathbb{R}^{3}, \|.\|_2).$  Then it can be shown using elementary arguments that $T\bot_B A $ but $A \not\perp_B T$.  \\ 

In \cite{S}, Sain introduced the notion of left symmetric and right symmetric points in Banach spaces, defined as follows:\\
\noindent \textbf{Left symmetric point:} An element $x\in \mathbb{X}$ is called left symmetric if $x\bot_By \Rightarrow y\bot_B x$ for all $y \in \mathbb{X}$.\\
\noindent \textbf{Right symmetric point:}  An element $x\in \mathbb{X}$ is called right symmetric if $y\bot_Bx \Rightarrow x\bot_B y$ for all $y \in \mathbb{X}$.\\
Let us say that an element $x\in \mathbb{X}$ is a symmetric point if $ x $ is both left symmetric and right symmetric. The following two notions, introduced in the same paper \cite{S}, are also relevant in context of our present work:\\
For any two elements $ x, y $ in a real normed linear space $ \mathbb{X}, $ let us say that $ y \in x^{+} $ if $ \| x + \lambda y \| \geq \| x \| $ for all $ \lambda \geq 0. $ Accordingly, we say that $ y \in x^{-} $ if $ \| x + \lambda y \| \geq \| x \| $ for all $ \lambda \leq 0. $\\

In \cite{GSP} we proved that if $\mathbb{H}$ is a real finite-dimensional Hilbert space, $T \in B(\mathbb{H})$ is right symmetric if and only if $M_T = S_{\mathbb{H}}$ and $T \in B(\mathbb{H})$ is left symmetric if and only if $T$ is the zero operator. It should be noted that if $ \mathbb{H} $ is a complex Hilbert space then Theorem $ 2.5 $ of \cite{T} gives a complete characterization of right symmetric bounded linear operators in $ B(\mathbb{H}), $ in terms of isometry and coisometry. However, these results are no longer true in general if we allow the operators to be defined on a Banach space instead of a Hilbert space. In fact, Example $ 1 $ in \cite{S} suffices to validate our remark. Sain proved in the same paper that a linear operator $T$ defined on the two-dimensional real $l_p(1< p < \infty)$ space is left symmetric if and only if $T$ is the zero operator. He also remarked in \cite{S} that it would be interesting to extend this result to higher dimensional $ l_p $ spaces, and more generally, to finite-dimensional strictly convex and smooth real Banach spaces, if possible.\\

In this paper we completely characterize left symmetric linear operators defined on a finite-dimensional strictly convex and smooth Banach space $ \mathbb{X}. $ We prove that $  T\in B(\mathbb{X}) $ is left symmetric if and only if $ T $ is the zero operator. It should be noted that if $ \mathbb{X} $ is two-dimensional then we can do away with the smoothness assumption, since only strict convexity is sufficient to obtain the desired characterization in this case. We also explore the right symmetry of Birkhoff-James orthogonality of linear operators defined on finite-dimensional Banach spaces. We show that if $ \mathbb{X} $ is a finite-dimensional strictly convex and smooth Banach space and  $ T \in B(\mathbb{X}) $ is a smooth point in $ B(\mathbb{X}) $ then $ T $ can not be right symmetric. Furthermore, when the underlying Banach space is not necessarily strictly convex or smooth, we prove two results involving right symmetric property of linear operators.

\section{Main results}

We begin this section with the promised characterization of left symmetric operator(s) defined on a two-dimensional strictly convex Banach space.
\begin{theorem}\label{two}
Let $\mathbb{X}$ be a two-dimensional strictly convex Banach space. Then $T \in B(\mathbb{X}) $ is left symmetric if and only if $T$ is the zero operator.
\end{theorem}
\begin{proof}
If possible suppose that $T$ is a non-zero left symmetric operator. Since $ \mathbb{X} $ is finite-dimensional, there exists $x_1 \in S_X$ such that $\|Tx_1\| = \|T\|$.\\
It follows from $ \mbox{Theorem}~ 2.3 $ of James \cite{J} that there exists $x_2 \in S_X$ such that $x_2 \perp_B x_1$. Furthermore, it follows from Theorem 2.5 of Sain \cite{S} that $Tx_2 = 0$.\\
We next claim that $ x_1 \perp_B x_2. $ \\
Once again, it follows from $ \mbox{Theorem}~ 2.3 $ of James \cite{J} that there exists a real number $ a $ such that $ ax_2 + x_1 \perp_B x_2$. Since $ x_1 \perp_B x_2 $ and $ x_1, x_2 \neq 0, $ $ \{ x_1, x_2 \} $ is linearly independent and hence $ ax_2 + x_1 \neq 0. $ Let $ z = \frac{ax_2 + x_1}{\| ax_2 + x_1 \|}. $ We note that if $ Tz = 0 $ then $ T $ is the zero operator. Let $ Tz \neq 0. $ Clearly, $ \{ x_2, z \} $ is a basis of $ \mathbb{X}, $ since $ \mathbb{X} $ is two-dimensional.\\
Let $ \| c_1 z + c_2 x_2 \| = 1, $ for some scalars $ c_1, c_2. $ Then we have, $ 1 = \| c_1 z + c_2 x_2 \| \geq ~\mid c_1 \mid. $ Since $ \mathbb{X} $ is strictly convex, $ 1 > \mid c_1 \mid, $ if $ c_2 \neq 0. $ We also have, $ \| T(c_1 z + c_2 x_2) \| = \| c_1 Tz \| = \mid c_1 \mid \| Tz \| \leq \| Tz \| $ and $ \| T(c_1 z + c_2 x_2) \| = \| Tz \| $ if and only if $ c_1 = \pm 1 $ and $ c_2 = 0. $ This proves that $ M_T = \{ \pm z \}. $ However, we have already assumed that $ x_1 \in M_T. $ Thus, we must have $ x_1 = \pm z. $ Since $ z \perp_B x_2, $ our claim is proved. Thus, $ x_1, x_2 \in S_{\mathbb{X}} $ are such that $ x_1 \perp_B x_2 $ and $ x_2 \perp_B x_1. $ \\
Let $u \in S_X$ such that $Tx_1 \perp_B u$.
By Theorem 2.4 of Sain \cite{S}, $Tx_1$ is a left symmetric point in $ \mathbb{X} $ and so $u \perp_B Tx_1$. By strict convexity of $\mathbb{X}$, we must have, $\|x_1 + x_2\| = 2-\delta$ for some $0 < \delta <1 $.\\
Choose $0<\epsilon< \frac{\delta}{3-\delta} $.\\
Let $v \in B(u,\epsilon)$ be such that $v = t_0 u + (1-t_0) Tx_1$, for some $t_0 \in (0,1)$. We may and do note that such a choice of $ v $ is always possible.\\
Define a linear operator $A$ as follows:
\begin{eqnarray*}
Ax_1 & = & u \\
Ax_2 & = & v 
\end{eqnarray*}
It is easy to verify that $T\perp_B A$, as $ x_1 \in M_{T} $ and $Tx_1 \perp_B Ax_1$.\\
Now, by virtue of our choice of $ \epsilon, $ we have,
\[ A\Big(\frac{x_1+x_2}{\|x_1+x_2\|}\Big) = \frac{\|u+v\|}{\|x_1+x_2\|} > \frac{2-\epsilon}{2-\delta} > 1+\epsilon \]
Since $\|A\| > 1 ,$ $x_1, x_2 \notin M_A$.
Let $z = -\alpha_1 x_1 + \alpha_2 x_2 \in S_{\mathbb{X}}$ be chosen arbitrarily, where $\alpha_1, \alpha_2 > 0$. Since $ \mathbb{X} $ is strictly convex, $ x_1 \perp_{B} x_2, $ $ x_2 \perp_{B} x_1, $ and $ z \in S_{\mathbb{X}}, $ it can be easily verified that $\alpha_1, \alpha_2 < 1$.\\
Now 
\[\|Az\| = \|(\alpha_2 - \alpha_1)u + \alpha_2 (v-u)\| < |\alpha_2 - \alpha_1|+|\alpha_2|\|v-u\| < 1+\epsilon < A\Big(\frac{x_1+x_2}{\|x_1+x_2\|}\Big)\]
 and so $z = (-\alpha_1 x_1 + \alpha_2 x_2 ) \notin M_A, $ where $\alpha_1, \alpha_2 > 0$ and $ z \in S_{\mathbb{X}} $.\\
By taking the symmetry of $ S_{\mathbb{X}} $ about the origin into consideration, this effectively proves the following: `` Let $z \in M_A$. Then $ z $ must be of the form $z = \alpha_1 x_1 + \alpha_2 x_2 $, where $\alpha_1, \alpha_2 $ are of same sign.''\\
 We further note that $ \alpha_1, \alpha_2 \neq 0, $ since $ x_1, x_2 \notin M_A. $\\
Let us first assume that $\alpha_1, \alpha_2 > 0 .$ We have, $Az = \alpha_1 u + \alpha_2 v $ and $ Tz= \alpha_1 Tx_1. $ \\
We claim that $Tz \notin (Az)^-$.  \\
From Proposition 2.1 of Sain \cite{S}, it is easy to observe that it is sufficient to show: $Tx_1 \notin (\alpha_1 u+\alpha_2 v)^-$.\\
Now, 
\[\alpha_1 u+\alpha_2 v = \alpha_1 u + \alpha_2 (t_0 u+(1-t_0)Tx_1) = (\alpha_1 + \alpha_2 t_0)u + \alpha_2(1-t_0) Tx_1 .\]
So, 
\[ \|\alpha_1 u+\alpha_2 v - \alpha_2(1-t_0) Tx_1\| = \|(\alpha_1 + \alpha_2 t_0)u\| < \|(\alpha_1 + \alpha_2 t_0)u + \alpha_2(1-t_0) Tx_1\| = \|\alpha_1 u+\alpha_1 v\| \]
$\Rightarrow Tx_1 \notin (\alpha_1 u+\alpha_2 v)^-$, as claimed.\\
Similarly, if $\alpha_1, \alpha_2 < 0 $, we can show that $Tz \notin (Az)^-$.\\
Since for all $z \in M_A$, $Tz \notin (Az)^-$, using Theorem 2.2 of Sain \cite{S}, we conclude that $A \not\perp_B T$, which contradicts our initial assumption that $T$ is a non-zero left symmetric operator.

\end{proof}

For the corresponding result on higher dimensional Banach spaces, we first need the following lemma. We would also like to remark this gives an alternative proof to the last part of Theorem $ 2.2 $ in \cite{Sa}.

\begin{lemma}
Let $ \mathbb{X} $ be a Banach space, $ T \in B(\mathbb{X}) $ and $ x \in M_T .$ If in addition, both $x$ and $ Tx $ are smooth points in $ \mathbb{X} $ then for any $ y \in \mathbb{X}, $ we have, $ x \perp_{B} y \Rightarrow Tx \perp_{B} Ty. $
\end{lemma}
\begin{proof}
Without any loss of generality we can assume that $\|T\| = 1$. Since $x$ is a smooth point, there exists a unique linear functional $f \in S_{\mathbb{X}^*}$ such that $f(x) = \|x\| = 1. $\\
Again since $Tx$ is a smooth point, there exists a unique linear functional $g \in S_{\mathbb{X}^*}$ such that $g(Tx) = \|Tx\| = \|T\|\|x\| = 1$.\\
Now $g\circ T$ is a linear functional on $\mathbb{X}$ and $\|g \circ T\| \leq \|g\|\|T\| = \|T\| = 1.$ So $\|g \circ T\| = 1.$ From the uniqueness of $f$ we get, $f = g \circ T .$\\
As $x \perp_B y$, we have $f(y) = 0, $ i.e., $g(Ty) = 0$. However, this is equivalent to $Tx \perp_B Ty,$ which completes the proof of the lemma. \\

\end{proof}

When the dimension of $ \mathbb{X} $ is strictly greater than $ 2, $ we have the following theorem regarding left symmetric linear operator(s) in $ B(\mathbb{X}). $ In this case we need the additional assumption of smoothness on $ \mathbb{X}. $ 

\begin{theorem}
Let $\mathbb{X}$ be an $n-$dimensional strictly convex and smooth Banach space. $T \in B(\mathbb{X}) $ is left symmetric if and only if $T$ is the zero operator.
\end{theorem}
\begin{proof}
If possible suppose that $T$ is a non-zero left symmetric operator. Since $ \mathbb{X} $ is finite-dimensional, there exists $x_1 \in S_X$ such that $\|Tx_1\| = \|T\|$.\\
We first claim that $x_1$ is right symmetric.\\
If possible suppose that $x_1$ is not right symmetric, i.e., there exists $y \in S_{\mathbb{X}}$ such that $y \perp_B x_1$, but $x_1 \not\perp_B y$.\\
Let $H$ be the hyperplane of codimension 1 such that $ y\perp_B H $. Then any $w \in \mathbb{X}$ can be written as $w = ay + h$ for some scalar $a$ and $ h \in H .$ Define a linear operator $A$ on $\mathbb{X}$ such that $Aw = aTx_1 $. Clearly, $M_A = \pm\{y\}$. Since $\mathbb{X}$ is smooth and $y\perp_B x_1,$ we have $x_1 \in H $, so $Ax_1 = 0$, from which it follows that $Tx_1 \perp_B Ax_1$. As $x_1 \in M_T$, $T\perp_B A$. 
Now $x_1 \not\perp_B y,$ and $x_1 \in M_T$, so by Proposition $ 2.1 $ of Sain \cite{Sa}, we get $Tx_1 \not\perp_B Ty$, i.e., $Ay \not\perp_B Ty.$ Since $M_A = \pm \{y\}$, by Theorem $2.1$ of Sain \cite{SP}, it follows that $A\not\perp_B T$, which contradicts that $T$ is left symmetric. Hence $x_1$ must be right symmetric.\\
We next claim that $x_1$ is left symmetric.\\
If possible suppose that $x_1$ is not left symmetric, i.e., there exists $z \in S_{\mathbb{X}}$ such that $x_1 \perp_B z$, but $z \not\perp_B x_1$.\\
We now prove that $Tz = 0$.\\
If possible suppose that $Tz \neq 0$. Let $H_z$ be the hyperplane of codimension 1 such that $ z\perp_B H_z $. Now, any $w \in \mathbb{X}$ can be written as $w = az + h$ for some scalar a and $ h \in H_z .$ Define a linear operator $A$ on $\mathbb{X}$ such that $Aw = aTz $. Since $ \mathbb{X} $ is strictly convex, $M_A = \pm\{z\}$. As $Az \not\perp_B Tz$, applying Theorem $ 2.1 $ of \cite{SP}, we conclude that $A \not\perp_B T$. Since $\mathbb{X}$ is smooth and $x_1 \perp_B z,$ applying Lemma $2.1$ we get, $Tx_1 \perp_B Tz $. It is easy to check that $Ax_1 = Tz$. So $Tx_1 \perp_B Ax_1$. Since $x_1 \in M_T$, we have, $T\perp_B A$. Thus we have, $ T\perp_B A $ but $A \not\perp_B T,$ which contradicts our assumption that $ T $ is left symmetric. This completes the proof of our claim.\\
Now, from Theorem 2.3 of James \cite{J}, it follows that there exists a scalar $k$ such that $kx_1+z \perp_B x_1$. As $T$ is left symmetric, by Theorem $2.5$ of Sain \cite{S}, we get $T(kx_1+z) = 0$. Since $ Tz=0 $ and $ Tx_1 \neq 0, $ it now follows that $k = 0$. So $z\perp_B x_1$, a contradiction to our choice of $ z. $ Therefore $x_1$ is left symmetric. Thus, combining these two observations, we conclude that $ x_1 $ is a  symmetric point in $ \mathbb{X}. $\\
Let $H_1$ be the subspace of codimension one such that $x_1\perp_B H_1$. Since $x_1$ is symmetric, $H_1 \perp_B x_1$ and by Theorem 2.5 of Sain \cite{S}, $T(H_1) = 0. $ Suppose that $\{x_2, x_3, \ldots , x_{n}\}$ is a basis of $H_1$ such that $x_2\perp_B$ span$\{x_3, x_4, \ldots ,x_{n}\}$. Since $\mathbb{X}$ is smooth, using Theorem 4.2 of \cite{J}, we conclude that $x_2 \perp_B $ span$\{x_1, x_3, \ldots ,x_{n}\}$.\\
Let $u \in S_{\mathbb{X}}$ such that $Tx_1 \perp_B u$.
By Theorem $2.4$ of Sain \cite{S}, $Tx_1$ is left symmetric and so $u \perp_B Tx_1$. As in the proof of Theorem $ 2.1, $ by strict convexity of $\mathbb{X}$, $\|x_1 + x_2\| = 2-\delta$ for some $0 < \delta <1 $.\\
As before, choose $0<\epsilon< \frac{\delta}{3-\delta} $ and let $v \in B(u,\epsilon)$ such that $v = t_0 u + (1-t_0) Tx_1$, for some $t_0 \in (0,1)$.\\
Define a linear operator $A$ on $ \mathbb{X} $ in the following way:
\begin{eqnarray*}
Ax_1 & = & u \\
Ax_2 & = & v \\
Ax_i & = & 0, n \geq i \geq 3
\end{eqnarray*}
It is easy to verify that $T\perp_B A$, as $x_1 \in M_T$ and $Tx_1 \perp_B Ax_1$.\\
Now following the same arguments as in the Theorem \ref{two}, we get,
\[ A\Big(\frac{x_1+x_2}{\|x_1+x_2\|}\Big)  
> 1+\epsilon \]
Since $\|A\| > 1 ,$ $x_i \notin M_A$ for all $n \geq i \geq 1$.
Let $z = -\alpha_1 x_1 + \alpha_2 x_2 + \ldots + \alpha_{n} x_{n} \in S_{\mathbb{X}}$, where $\alpha_1, \alpha_2 > 0$. Since $ \mathbb{X} $ is strictly convex, $x_1 \perp_B $ span$\{x_2, x_3, \ldots ,x_{n}\} ,$ $x_2 \perp_B $ span$\{x_1, x_3, \ldots ,x_{n}\} ,$ and $ z \in S_{\mathbb{X}}, $ it can be easily verified that $\alpha_1, \alpha_2 < 1$.\\
Since 
\[\|Az\| = \|(\alpha_2 - \alpha_1)u + \alpha_2 (v-u)\| < |\alpha_2 - \alpha_1|+|\alpha_2|\|v-u\| < 1+\epsilon < A\Big(\frac{x_1+x_2}{\|x_1+x_2\|}\Big)\], we may conclude that,
$z = (-\alpha_1 x_1 + \alpha_2 x_2 + \ldots + \alpha_{n} x_{n}) \notin M_A $, where $\alpha_1, \alpha_2 > 0$ and $ z \in S_{\mathbb{X}}. $\\
Similar to the proof of Theorem $ 2.1, $ we thus have the following conclusion:\\
Let $z \in M_A$. Then $ z $ must be of the form $z = (\alpha_1 x_1 + \alpha_2 x_2+ \ldots + \alpha_{n} x_{n}) $, where $\alpha_1, \alpha_2 $ are of same sign. We further note that $ \alpha_1, \alpha_2 \neq 0, $ since $ x_1, x_2 \notin M_A. $\\ 
Next we may proceed in the same way as in Theorem \ref{two}, to show that $Tz \notin (Az)^-$. This, along with Theorem 2.2 of Sain \cite{S}, lead to the conclusion that $A\not\perp_B T .$ This proves that $ T $ is not a left symmetric point in $ B(\mathbb{X}) $ and completes the proof of the theorem.  \\

\end{proof}

In the next theorem we prove that smooth linear operators defined on a finite-dimensional strictly convex and smooth Banach space can not be right symmetric.

\begin{theorem}
Let $\mathbb{X}$ be a finite-dimensional strictly convex and smooth Banach space. Let $T \in B(\mathbb{X}) $ be smooth. Then $T$ is not right symmetric.
\end{theorem}

\begin{proof}
If possible suppose that $T$ is a right symmetric operator on $\mathbb{X} $ and $ T $ is a smooth point in $ B(\mathbb{X}). $ We first note that since $ T $ is smooth, it follows from Theorem $ 4.2 $ of \cite{P} that $ M_T=\{\pm x\}, $ for some $ x \in S_{\mathbb{X}}. $ We claim that $x$ is left symmetric.\\
If possible, suppose that $x$ is not left symmetric. Then there exists $y$ such that $x \perp_B y$ but $y \not\perp_B x$. Let $H_y$ be the hyperplane of codimension one such that $y \perp_B H_y$. Clearly, any $z \in \mathbb{X}$ can be written as $z = \alpha y + h$ for some $h \in H_y$.\\
Define a linear operator $A$ on  $\mathbb{X}$ such that $A(\alpha y +h) = \alpha Tx$.\\
It is easy to show that $M_A = \pm \{y\}$. Clearly, $A \perp_B T$, since $Ay \perp_B Ty$. \\
Since $y \not\perp_B x$, by Proposition $2.1$ of Sain \cite{Sa} we have $Ay \not\perp_B Ax$, i.e., $Tx \not\perp_B Ax$. Since $ M_T=\{\pm x\}, $ it follows from Theorem $ 2.1 $ of \cite{SP}, that $T \not\perp_B A, $ which contradicts that $ T $ is right symmetric. Therefore we must have that $x$ is left symmetric.\\
Let $H_x$ be the hyperplane of codimension one such that $x\perp_B H_x$. Since $x$ is left symmetric, $H_x \perp_B x$.
Consider the point $z = x+h_0$, where $h_0 \in H_x $ such that $\|Th_0\| > \|T\|$. Take $ z' = \frac{z}{\|z\|}$. Since $\mathbb{X}$ is strictly convex, by Theorem $4.3$ of James \cite{J}, Birkhoff-James orthogonality is left unique. Since $h_0 \perp_B x$, $z \not\perp_B x$.\\
Now, there exists a scalar $d$ such that $(dTz'+Th_0) \perp_B Tz'$.\\
We next claim that $d \neq 0$.\\
If $d=0$, then $Th_0 \perp_B Tz'$, from which it follows that $Th_0 \perp_B (Tx+Th_0)$.\\
But 
\[ \|Th_0-(Tx+Th_0)\| = \|Tx\| = \|T\| < \|Th_0\|, \]
which contradicts that $Th_0 \perp_B (Tx+Th_0)$. So $d \neq 0$.\\
Let $H_z$ be the hyperplane of codimension one such that $z' \perp_B H_z$. Then any $w \in \mathbb{X}$ can be written as $w = \alpha z' + h'$ for some $h' \in H_z$.\\
Define a linear operator $A$ on  $\mathbb{X}$ such that 
\[A(\alpha z' + h') =\alpha (dTz' + Th_0). \]
It is easy to show that $M_A = \pm \{z'\}$. Clearly, $A \perp_B T$, since $Az' \perp_B Tz'$. \\
We prove that $T \not\perp_B A$.\\
If $T \perp_B A$, using Theorem $2.1$ of Sain and Paul \cite{SP} it follows that $Tx \perp_B Ax $.\\
Now $ x = \alpha z' + h'$ for some $h' \in H_z$. It is easy to check that $\alpha \neq 0$.\\
So $Ax = \alpha Az' = \alpha (dTz' + Th_0)$.\\
As $Tx \perp_B (dTz'+Th_0)$ and $x \in M_T$ , by Proposition $2.1$ of \cite{Sa}, we have, $x\perp_B (dz'+h_0)$. Now $x \perp_B h_0$ and $\mathbb{X}$ is smooth. So $x \perp_B z'$. Also since $x$ is left symmetric and Birkhoff-James orthogonality is homogeneous, we have $z \perp_B x$, which contradicts that $z \not\perp_B x$.\\
Hence $T \not\perp_B A$. This proves that $T$ is not a right symmetric operator.
\end{proof}

\noindent When $ \mathbb{X} $ is not necessarily strictly convex or smooth, we have the following two theorems regarding right symmetric operators.
\begin{theorem}
Let $ \mathbb{X} $ be an $ n- $dimensional Banach space. Let $ x_0 \in S_{\mathbb{X}} $ be a left symmetric point. Let $ T \in B(\mathbb{X})  $ be such that $ M_T = \{ \pm x_0 \} $ and $ x_0 $ is an eigen vector of $ T. $ Then either of the following is true: \\
\noindent (i) $ rank~ T \geq n-1. $ \\
\noindent (ii) $ T $ is not a right symmetric point in $ B(\mathbb{X}). $
\end{theorem}

\begin{proof} We first note that the theorem is trivially true if $ n \leq 2. $ Let $ n > 2. $ Since $ x_0 $ is an eigen vector of $ T, $ there exists a scalar $ \lambda_0 $ such that $ Tx_0 = \lambda_0 x_0. $ We also note that since $ M_T = \{ \pm x_0 \}, \lambda_0 \neq 0. $ If  $ rank~ T $ $ \geq n-1 $ then we are done. Let $ rank~ T < n-1. $ Then $ ker~ T $ is a subspace of $ \mathbb{X} $ of dimension at least $ 2. $ Let $ x_0 \perp_B H_0, $ where $ H_0 $ is a hyperplane of codimension $ 1 $ in $ \mathbb{X}. $ Since $ dim~ker~T \geq 2, $ there exists a unit vector $ u_0 \in S_{\mathbb{X}} $ such that $ u_0 \in H_0 \cap ker~T. $ Since $ x_0 $ is a left symmetric point and $ x_0 \perp_B u_0, $ we have $ u_0 \perp_B x_0. $ There exists a hyperplane $ H_1 $ of codimension $ 1 $ in $ \mathbb{X} $ such that $ u_0 \perp_B H_1 $ and $ x_0 \in H_1. $ Let $ \{ x_0, y_i : i = 1, 2, \ldots, n-2\} $ be a  basis of $ H_1. $ Then $ \{ u_0, x_0, y_i : i = 1, 2, \ldots, n-2 \} $ is  basis of $ \mathbb{X} $ such that $ u_0 \perp_B span\{ x_0, y_i : i = 1, 2, \ldots, n-2 \}. $ Define a linear operator $ A \in B(\mathbb{X}) $ as follows: \\
\[  Au_0 = u_0, Ax_0 = \frac{1}{2} x_0, Ay_i = \frac{1}{2} y_i. \]

It is routine to check that $ u_0 \in M_A  $.  Since $ Au_0 \perp_B Tu_0 ,$ 
$ A \perp_B T. $ However, since $ \lambda_0 \neq 0, Tx_0 = \lambda_0 x_0 \not\perp_B \frac{1}{2} x_0 = Ax_0. $ This, coupled with the fact that $ M_T = \{ \pm x_0 \}, $ implies that $ T \not\perp_B A $ and thus $ T $ is not a right symmetric point in $ B(\mathbb{X}). $
\end{proof}

\begin{theorem}
Let $ \mathbb{X} $ be an $ n- $dimensional Banach space. Let $ T \in B(\mathbb{X}) $ be such that $ M_T = \{ \pm x_0 \} $ and $ ker~T $ contains a non-zero left symmetric point. Then either of the following is true: \\
\noindent (i) $ I \perp_B T $ and $ T \perp_B I, $ where $ I \in B(\mathbb{X}) $ is the identity operator on $ \mathbb{X}. $ \\
\noindent (ii) $ T $ is not a right symmetric point in $ B(\mathbb{X}). $
\end{theorem}

\begin{proof} Let $ u_0 \in ker~T $ be a non-zero left symmetric point. Without loss of generality let us assume that $ \| u_0 \| = 1. $ We have, $ \| I + \lambda T \| \geq \| (I + \lambda T) u_0 \| = 1 \geq \| I \|, $ which proves that $ I \perp_B T. $ If $ T \perp_B I $ then we are done. If possible suppose that $ T \not\perp_B I. $ Since $ M_T = \{ \pm x_0 \}, $ it follows that $ Tx_0 \not\perp_B Ix_0 =x_0. $ 
Let $ H_0 $ be a hyperplane of codimension $ 1 $ in $ \mathbb{X} $ such that $ u_0 \perp_B H_0. $ Let $ \{u_1, u_2, \ldots, u_{n-1}\} $ be a  basis of $ H_0. $ Then $ \{u_0, u_1, \ldots, u_{n-1}\} $ is a basis of $ \mathbb{X} $ such that $ u_0 \perp_B span\{ u_1, u_2, \ldots, u_{n-1} \}. $ \\
\noindent Let $ x_0 = \alpha_0 u_0 + \alpha_1 u_1 + \ldots + \alpha_{n-1} u_{n-1}, $ for some scalars $ \alpha_0, \alpha_1, \ldots, \alpha_{n-1}.  $ 
Clearly, we have $ u_0 \perp_B \alpha_1 u_1 + \ldots + \alpha_{n-1} u_{n-1}. $ Since $ u_0 $ is a left symmetric point in $ \mathbb{X}, $  $ \alpha_1 u_1 + \ldots + \alpha_{n-1} u_{n-1} \perp_B u_0.$ \\
We claim that $ \alpha_0 = 0. $ We have, \\
$ 1 = \| x_0 \| = \| \alpha_0 u_0 + (\alpha_1 u_1 + \ldots + \alpha_{n-1} u_{n-1}) \| = \|  (\alpha_1 u_1 + \ldots + \alpha_{n-1} u_{n-1}) + \alpha_0 u_0 \| \geq \| (\alpha_1 u_1 + \ldots + \alpha_{n-1} u_{n-1}) \|.$  We also have, $ \| T(\alpha_1 u_1 + \ldots + \alpha_{n-1} u_{n-1}) \| = \| T(\alpha_0 u_0 + \alpha_1 u_1 + \ldots + \alpha_{n-1} u_{n-1}) \| = \| Tx_0 \| = \| T \|$. This proves that $ (\alpha_1 u_1 + \ldots + \alpha_{n-1} u_{n-1}) \in M_T. $ Since $ M_T = \pm \{ x_0 \}, $ we must have  $x_0 =  (\alpha_1 u_1 + \ldots + \alpha_{n-1} u_{n-1})$ or $x_0 =  -(\alpha_1 u_1 + \ldots + \alpha_{n-1} u_{n-1}).$ Since $ x_0 = \alpha_0 u_0 + \alpha_1 u_1 + \ldots + \alpha_{n-1} u_{n-1} \in S_{\mathbb{X}}, $  this shows that $ x_0 =  (\alpha_1 u_1 + \ldots + \alpha_{n-1} u_{n-1}) $ and $ \alpha_0 = 0. $ \\
Thus we have, $ x_0 = \alpha_1 u_1 + \ldots + \alpha_{n-1} u_{n-1} $ and $ u_0 \perp_B x_0. $ 
Let $ \{ x_0, y_i : i = 3, 4, \ldots , n \} $ be a basis of $ H_0 $. Then $ \{ u_0, x_0, y_i : i = 3, 4, \ldots , n \} $ is a basis of $ \mathbb{X} $ such that $ u_0 \perp_B span\{ x_0, y_i : i = 3, 4, \ldots , n \} $. Define a linear operator $ A \in B(\mathbb{X}) $ as follows: 
\[ Au_0 = u_0, Ax_0 = \frac{1}{2} x_0, Ay_i = \frac{1}{2} y_i. \]
As before, it is easy to check that $ u_0 \in M_A . $ Clearly, $ A \perp_B T, $ since $ Au_0 \perp_B Tu_0 = 0. $ We also note that since $ M_T = \{ \pm x_0 \} $ and $ Tx_0 \not\perp_B Ax_0 = \frac{1}{2} x_0, $ we must have that $ T \not\perp_B A. $ This proves that $ T $ is not a right symmetric point in $ B(\mathbb{X}) $ and completes the proof of the theorem.
\end{proof}

In view of the results obtained in the present paper, we would like to end it with the remark that obtaining a characterization of right symmetric linear operators defined on a  finite-dimensional strictly convex and smooth Banach space, seems to be a very interesting problem. It should be noted that for complex Hilbert spaces, right symmetric bounded linear operators are characterized by isometries or coisometries \cite{T}.

\bibliographystyle{amsplain}

\end{document}